\documentclass{article}

\usepackage{amsmath,amssymb,amsfonts,amsthm}
\usepackage{graphicx}
\usepackage{subfig}
\usepackage{url}
\usepackage[english]{babel}
\usepackage{calrsfs}
\usepackage{a4wide}
\usepackage{cite}

\newtheorem{theorem}{Theorem}
\newtheorem{lemma}[theorem]{Lemma}
\newtheorem{proposition}[theorem]{Proposition}

\newtheorem{remark}[theorem]{Remark}

\def\R{\mathbb{R}}

\begin{document}

\title{An epidemic model for cholera\\ 
with optimal control treatment\thanks{This work 
is part of first author's Ph.D. project, 
which is carried out at the University of Aveiro.\newline
This is a preprint of a paper whose final and definite form is with 
\emph{Journal of Computational and Applied Mathematics}, ISSN 0377-0427,
available at {\tt http://dx.doi.org/10.1016/j.cam.2016.11.002}.
Submitted 19-June-2016; Revised 14-Sept-2016; Accepted 04-Nov-2016.}}
 	
\author{Ana P. Lemos-Pai\~{a}o\\
\texttt{anapaiao@ua.pt}
\and
Cristiana J. Silva\\
\texttt{cjoaosilva@ua.pt}
\and
Delfim F. M. Torres\\
\texttt{delfim@ua.pt}}

\date{Center for Research and Development 
in Mathematics and Applications (CIDMA)\\
Department of Mathematics,
University of Aveiro, 
3810-193 Aveiro, Portugal}

\maketitle


\begin{abstract}
We propose a mathematical model for cholera with treatment through quarantine. 
The model is shown to be both epidemiologically and mathematically well posed.
In particular, we prove that all solutions of the model are positive 
and bounded; and that every solution with initial conditions in a certain 
meaningful set remains in that set for all time. The existence of unique 
disease-free and endemic equilibrium points is proved and the basic reproduction 
number is computed. Then, we study the local asymptotic stability 
of these equilibrium points. An optimal control problem is proposed and analyzed, 
whose goal is to obtain a successful treatment through quarantine. We provide 
the optimal quarantine strategy for the minimization of the number 
of infectious individuals and bacteria concentration, as well as 
the costs associated with the quarantine. Finally, a numerical simulation 
of the cholera outbreak in the Department of Artibonite (Haiti), in 2010, 
is carried out, illustrating the usefulness of the model and its analysis.\newline

\noindent {\bf Keywords:} SIQRB cholera model, basic reproduction number, 
disease-free and endemic equilibria, local asymptotic stability, 
optimal control, numerical case study of Haiti.\newline

\noindent {\bf Mathematics Subject Classification 2010:} 34C60, 49K15, 92D30.
\end{abstract}


\section{Introduction}

\indent 

Cholera is an acute diarrhoeal illness caused by infection of
the intestine with the bacterium \textit{Vibrio cholerae}, 
which lives in an aquatic organism. The ingestion of contaminated water
can cause cholera outbreaks, as John Snow proved, in 1854 \cite{Shuai}.
This is a way of transmission of the disease, but others exist. 
For example, susceptible individuals can become infected 
if they contact with infected individuals. If individuals 
are at an increased risk of infection, then they can transmit 
the disease to other persons that live with them by reflecting 
food preparation or using water storage containers \cite{Shuai}. 
An individual can be infected with or without symptoms.
Some symptoms are watery diarrhoea, vomiting and leg cramps. 
If an infected individual does not have treatment, 
then he becomes dehydrated, suffering of acidosis and circulatory
collapse. This situation can lead to death within 12 to 24 hours \cite{Mwasa,Shuai}.
Some studies and experiments suggest that a recovered individual can be immune
to the disease during a period of 3 to 10 years. Recent researches suggest
that immunity can be lost after a period of weeks to months \cite{Neilan,Shuai}.
 
Since 1979, several mathematical models for the 
transmission of cholera have been proposed: see, e.g.,
\cite{Capasso:1979,Capone,Codeco:2001,Hartley:2006,Hove-Musekwa,%
Joh:2009,Mukandavire:2008,Mwasa,Neilan,Pascual,Shuai} 
and references cited therein. In \cite{Neilan}, 
the authors propose a SIR (Susceptible--Infectious--Recovered)
type model and consider two classes of bacterial concentrations 
(hyperinfectious and less-infectious) and two classes 
of infectious individuals (asymptomatic and symptomatic). 
In \cite{Shuai}, another SIR-type model is proposed 
that incorporates, using distributed delays, 
hyperinfectivity (where infectivity varies with the time since the pathogen was shed)
and temporary immunity. The authors of \cite{Mwasa} incorporate in a SIR-type model 
public health educational campaigns, vaccination, quarantine and treatment, 
as control strategies in order to curtail the disease. 

The use of quarantine for controlling epidemic diseases 
has always been controversial, because such strategy raises political, 
ethical and socioeconomic issues and requires a careful balance between 
public interest and individual rights \cite{Tognotti:quarantine}. 
Quarantine was adopted as a mean of separating persons, animals 
and goods that may have been exposed to a contagious disease. Since 
the fourteenth century, quarantine has been the cornerstone 
of a coordinated disease-control strategy, including isolation, 
sanitary cordons, bills of health issued to ships, fumigation, 
disinfection and regulation of groups of persons who were believed 
to be responsible for spreading of the infection \cite{Matovinovic,Tognotti:quarantine}.
The World Health Organization (WHO) does not recommend quarantine measures 
and embargoes on the movement of people and goods for cholera. However, 
cholera is still on the list of quarantinable diseases of the EUA National 
Archives and Records Administration \cite{CenterDiseaseControl}.
In this paper, we propose a SIQR (Susceptible--Infectious--Quarantined--Recovered)
type model, where it is assumed that infectious individuals are subject 
to quarantine during the treatment period. 

Optimal control is a branch of mathematics developed
to find optimal ways to control a dynamic system
\cite{Cesari_1983,Fleming_Rishel_1975,Pontryagin_et_all_1962}.
There are few papers that apply optimal control to cholera models \cite{Neilan}. 
Here we propose and analyze one such optimal control problem, where the control 
function represents the fraction of infected individuals $I$ that will 
be submitted to treatment in quarantine until complete recovery. The objective 
is to find the optimal treatment strategy through quarantine that minimizes 
the number of infected individuals and the bacterial concentration, 
as well as the cost of interventions associated with quarantine. 

Between 2007 and 2011, several cholera outbreaks occurred,
namely in Angola, Haiti and Zimbabwe \cite{Shuai}. In Haiti, the first 
cases of cholera happened in Artibonite Department on 14th 
October 2010. The disease propagated along the Artibonite river 
and reached several departments. Only within one month, 
all departments had reported cases in rural areas and places without 
good conditions of public health \cite{WHO}. In this paper, we provide 
numerical simulations for the cholera outbreak in the Department 
of Artibonite, from $1$st November 2010 until $1$st May 2011 \cite{WHO}. 
Our work is of great significance, because it provides an approach to cholera
with big positive impact on the number of infected individuals and on the bacterial
concentration. This is well illustrated with the real data of the cholera outbreak 
in Haiti that occurred in 2010. More precisely,
we show that the number of infectious individuals decreases
significantly and that the bacterial concentration 
is a strictly decreasing function, when our control strategy is applied. 

The paper is organized as follows. In Section~\ref{Sec:model}, 
we formulate our model for cholera transmission dynamics.
We analyze the positivity and boundedness of the solutions, 
as well as the existence and local stability of the disease-free 
and endemic equilibria, and we compute the basic reproduction number 
in Section~\ref{Sec:mod:analysis}. In Section~\ref{sec:ocp}, 
we propose and analyze an optimal control problem. 
Section~\ref{sec:num:simu} is devoted to numerical simulations. 
We end with Section~\ref{sec:conc}, by deriving some conclusions 
about the inclusion of quarantine in treatment. 


\section{Model formulation}
\label{Sec:model}

\indent 

We propose a SIQR (Susceptible--Infectious--Quarantined--Recovered) 
type model and consider a class of bacterial concentration for the 
dynamics of cholera. The total population $N(t)$ is divided into 
four classes: susceptible $S(t)$, infectious with symptoms $I(t)$, 
in treatment through quarantine $Q(t)$ and recovered $R(t)$ at time $t$, 
for $t \ge 0$. Furthermore, we consider a class $B(t)$ that reflects 
the bacterial concentration at time $t$. We assume that there is a positive 
recruitment rate $\Lambda$ into the susceptible class $S(t)$ and a positive 
natural death rate $\mu$, for all time $t$ under study. Susceptible
individuals can become infected with cholera at rate $\frac{\beta B(t)}{\kappa+B(t)}$
that is dependent on time $t$. Note that $\beta>0$ is the ingestion 
rate of the bacteria through contaminated sources, $\kappa$ is the half 
saturation constant of the bacteria population and $\frac{B(t)}{\kappa+B(t)}$ 
is the possibility of an infected individual to have the disease with symptoms, 
given a contact with contaminated sources \cite{Mwasa}. Any recovered individual 
can lose the immunity at rate $\omega$ and therefore becomes susceptible again. 
The infected individuals can accept to be in quarantine during a period of time. 
During this time they are isolated and subject to a proper medication, 
at rate $\delta$. The quarantined individuals can recover at rate
$\varepsilon$. The disease-related death rates associated with the individuals
that are infected and in quarantine are $\alpha_1$ and $\alpha_2$, respectively.
Each infected individual contributes to the increase of the bacterial concentration
at rate $\eta$. On the other hand, the bacterial concentration can decrease
at mortality rate $d$. These assumptions are translated in the following
mathematical model:
\begin{align}
\label{ModeloColera}
\begin{cases}
S'(t)=\Lambda-\displaystyle\frac{\beta B(t)}{\kappa+B(t)}S(t)+\omega R(t)-\mu S(t),\\
I'(t)=\displaystyle\frac{\beta B(t)}{\kappa+B(t)}S(t)-(\delta+\alpha_1+\mu)I(t),\\
Q'(t)=\delta I(t)-(\varepsilon+\alpha_2+\mu)Q(t),\\
R'(t)=\varepsilon Q(t)-(\omega+\mu)R(t),\\
B'(t)=\eta I(t)-dB(t).
\end{cases}
\end{align}


\section{Model analysis}
\label{Sec:mod:analysis}

\indent 

Throughout the paper, we assume that the initial conditions 
of system \eqref{ModeloColera} are nonnegative:
\begin{equation}
\label{eq:init:cond}
S(0) = S_0 \geq 0 \, , \quad I(0) = I_0 \geq 0 \, , 
\quad Q(0)= Q_0 \geq 0 \, , \quad R(0) = R_0 \geq 0 \, , 
\quad B(0) = B_0 \geq 0 \, .
\end{equation}


\subsection{Positivity and boundedness of solutions}

\indent 

Our first lemma shows that the considered model 
\eqref{ModeloColera}--\eqref{eq:init:cond} 
is biologically meaningful.

\begin{lemma}
The solutions $(S(t), I(t), Q(t), R(t), B(t))$ of system \eqref{ModeloColera} 
are nonnegative for all $t \geq 0$ with nonnegative initial conditions 
\eqref{eq:init:cond} in $(\R_0^+)^5$. 
\end{lemma}

\begin{proof}
We have
\begin{equation*}
\begin{cases}
\frac{d S(t)}{d t}\Bigg|_{\xi(S)} &= \Lambda + \omega R(t) > 0 \, , \\
\frac{d I(t)}{d t}\Bigg|_{\xi(I)} &= \frac{\beta B(t)}{\kappa+B(t)}S(t) > 0 \, , \\
\frac{d Q(t)}{d t}\Bigg|_{\xi(Q)} &= \delta I(t) > 0 \, , \\
\frac{d R(t)}{d t}\Bigg|_{\xi(R)} &= \varepsilon Q(t) > 0 \, , \\
\frac{d B(t)}{d t}\Bigg|_{\xi(B)} &= \eta I(t) > 0 \, ,
\end{cases}
\end{equation*}
where $\xi(\upsilon)=\left\{\upsilon(t)=0 \text{ and } 
S, I, Q, R, B \in C(\R_0^+,\mathbb{R}_0^+)\right\}$
and $\upsilon \in \{S, I, Q, R, B\}$.
Therefore, due to Lemma~2 in \cite{Yang:CMA:1996}, 
any solution of system \eqref{ModeloColera} is such that 
$(S(t), I(t), Q(t), R(t), B(t)) \in (\R_0^+)^5$
for all $t \geq 0$.  
\end{proof}

Next Lemma~\ref{lema:Inv} shows that it is sufficient 
to consider the dynamics of the flow generated by 
\eqref{ModeloColera}--\eqref{eq:init:cond} 
in a certain region $\Omega$. 

\begin{lemma}
\label{lema:Inv}
Let 
\begin{equation}
\label{eqN:1}
\Omega_H = \left\{ (S, I, Q, R) \in \left(\R_0^+\right)^4 \, | \, 0 
\leq S(t) + I(t) + Q(t) + R(t) \leq \frac{\Lambda}{\mu} \right\} 
\end{equation}
and
\begin{equation}
\label{eq:maj:B}
\Omega_B = \left\{ B \in \R_0^+ \, | \, 0 \leq B(t)  
\leq \frac{\Lambda\eta}{\mu d} \right\}.
\end{equation}
Define 
\begin{equation}
\label{eq:Omega}
\Omega = \Omega_H \times \Omega_B.
\end{equation}
If $N(0) \leq \frac{\Lambda}{\mu}$ and $B(0) \leq \frac{\Lambda\eta}{\mu d}$, 
then the region $\Omega$ is positively invariant for model \eqref{ModeloColera} 
with nonnegative initial conditions \eqref{eq:init:cond} in $(\R_0^+)^5$.
\end{lemma}

\begin{proof}
Let us split system \eqref{ModeloColera} into two parts: 
the human population, i.e., $S(t)$, $I(t)$, $Q(t)$ and $R(t)$, 
and the pathogen population, i.e., $B(t)$.
Adding the first four equations of system \eqref{ModeloColera} gives
\begin{equation*}
N'(t) = S'(t)+I'(t)+Q'(t)+R'(t)=\Lambda-\mu N(t)
-\alpha_1 I(t)-\alpha_2 Q(t) \leq \Lambda-\mu N(t) \, .
\end{equation*}
Assuming that $N(0) \leq \frac{\Lambda}{\mu}$, we conclude
that $N(t) \leq  \frac{\Lambda}{\mu}$. For this reason,
\eqref{eqN:1} defines the biologically feasible region 
for the human population. For the pathogen population, 
it follows that 
\begin{equation*}
B'(t) = \eta I(t)- d B(t) \leq \eta \frac{\Lambda}{\mu} - d B(t) \, . 
\end{equation*}
If $B(0) \leq \frac{\Lambda\eta}{\mu d}$, then $B(t) \leq \frac{\Lambda\eta}{\mu d}$
and, in agreement, \eqref{eq:maj:B} defines the biologically feasible region 
for the pathogen population. From \eqref{eqN:1} and \eqref{eq:maj:B}, we know 
that $N(t)$ and $B(t)$ are bounded for all $t \geq 0$. Therefore, every solution 
of system \eqref{ModeloColera} with initial condition 
in $\Omega$ remains in $\Omega$.
\end{proof}

In region $\Omega$ defined by \eqref{eq:Omega}, 
our model is epidemiologically and mathematically well posed 
in the sense of \cite{Hethcote}. In other words, every solution 
of the model \eqref{ModeloColera} with initial conditions in 
$\Omega$ remains in $\Omega$ for all $t \geq 0$.


\subsection{Equilibrium points and stability analysis}
\label{sec:3.2}

\indent 

The disease-free equilibrium (DFE) of model \eqref{ModeloColera} is given by
\begin{equation}
\label{eq:DFE}
E^0=(S^0,I^0,Q^0,R^0,B^0)=\left(\frac{\Lambda}{\mu},0,0,0,0\right).
\end{equation}
Next, following the approach of \cite{Mwasa,Driessche},
we compute the basic reproduction number $R_0$.

\begin{proposition}[Basic reproduction number of \eqref{ModeloColera}]
\label{prop:R0}
The basic reproduction number of model \eqref{ModeloColera} is given by
\begin{align}
\label{R0}
R_0=\frac{\beta\Lambda\eta}{\mu\kappa d(\delta+\alpha_1+\mu)}.
\end{align}
\end{proposition}

\begin{proof}
Consider that $\mathcal{F}_i(t)$ is the rate of appearance 
of new infections in the compartment associated with index $i$, 
$\mathcal{V}_i^+(t)$ is the rate of transfer of ``individuals'' 
into the compartment associated with index $i$ by all other means 
and $\mathcal{V}_i^-(t)$ is the rate of transfer of ``individuals'' 
out of compartment associated with index $i$. In this way, the matrices 
$\mathcal{F}(t)$, $\mathcal{V}^+(t)$ and $\mathcal{V}^-(t)$, 
associated with model \eqref{ModeloColera}, are given by
\begin{align*}
\mathcal{F}(t)=\left[
\begin{matrix}
0\\
\displaystyle\frac{\beta B(t)S(t)}{\kappa+B(t)}\\
0\\
0\\
0
\end{matrix}
\right],
\quad \mathcal{V}^+(t)=\left[
\begin{matrix}
\Lambda+\omega R(t)\\
0\\
\delta I(t)\\
\varepsilon Q(t)\\
\eta I(t)
\end{matrix}
\right]
\quad \text{ and } \quad
\mathcal{V}^-(t)=\left[
\begin{matrix}
\displaystyle\frac{\beta B(t)S(t)}{\kappa+B(t)}+\mu S(t)\\
a_1I(t)\\
a_2Q(t)\\
a_3R(t)\\
d B(t)
\end{matrix}\right],
\end{align*}
where 
\begin{equation}
\label{eq:as}
a_1=\delta+\alpha_1+\mu, 
\quad a_2=\varepsilon+\alpha_2+\mu 
\ \text{ and } \ a_3=\omega+\mu.
\end{equation}
Therefore, considering $\mathcal{V}(t)=\mathcal{V}^-(t)-\mathcal{V}^+(t)$,
we have that 
\begin{align*}
\left[
\begin{matrix}
S'(t) & I'(t) & Q'(t) & R'(t) & B'(t)
\end{matrix}\right]^T=\mathcal{F}(t)-\mathcal{V}(t).
\end{align*}
The Jacobian matrices of $\mathcal{F}(t)$ 
and of $\mathcal{V}(t)$ are, respectively, given by
\begin{align*}
F=\left[\begin{matrix}
0 & 0 & 0 & 0 & 0 \\
\displaystyle\frac{\beta B(t)}{\kappa+B(t)} 
& 0 & 0 & 0 & 
\displaystyle\frac{\beta\kappa S(t)}{(\kappa+B(t))^2} \\
0 & 0 & 0 & 0 & 0 \\
0 & 0 & 0 & 0 & 0 \\
0 & 0 & 0 & 0 & 0 \\
\end{matrix}\right]
\end{align*}
and
\begin{align*}
V=\left[\begin{matrix}
\displaystyle\frac{\beta B(t)}{\kappa+B(t)}+\mu 
& 0 & 0 & -\omega & 
\displaystyle\frac{\beta\kappa S(t)}{(\kappa+B(t))^2} \\
0 & a_1 & 0 & 0 & 0 \\
0 & -\delta & a_2 & 0 & 0 \\
0 & 0 & -\varepsilon & a_3 & 0 \\
0 & -\eta & 0 & 0 & d \\
\end{matrix}\right].
\end{align*}
In the disease-free equilibrium $E^0$ \eqref{eq:DFE},  
we obtain the matrices $F_0$ and $V_0$ given by
\begin{align*}
F_0=\left[\begin{matrix}
0 & 0 & 0 & 0 & 0 \\
0 & 0 & 0 & 0 & \displaystyle\frac{\beta \Lambda}{\mu\kappa} \\
0 & 0 & 0 & 0 & 0 \\
0 & 0 & 0 & 0 & 0 \\
0 & 0 & 0 & 0 & 0 \\
\end{matrix}\right]
\quad \text{ and } \quad
V_0=\left[\begin{matrix}
\mu & 0 & 0 & -\omega & \displaystyle\frac{\beta \Lambda}{\mu\kappa} \\
0 & a_1 & 0 & 0 & 0 \\
0 & -\delta & a_2 & 0 & 0 \\
0 & 0 & -\varepsilon & a_3 & 0 \\
0 & -\eta & 0 & 0 & d \\
\end{matrix}\right].
\end{align*}
The basic reproduction number 
of model \eqref{ModeloColera} is then given by
\begin{align*}
R_0=\rho(F_0V_0^{-1})=\frac{\beta\Lambda\eta}{\mu\kappa da_1}
=\frac{\beta\Lambda\eta}{\mu\kappa d(\delta+\alpha_1+\mu)}.
\end{align*}
This concludes the proof.
\end{proof}

Now we prove the local stability 
of the disease-free equilibrium $E^0$. 

\begin{theorem}[Stability of the DFE \eqref{eq:DFE}] 
The disease-free equilibrium $E^0$ of model \eqref{ModeloColera} is
\begin{enumerate}
\item Locally asymptotic stable, 
if $\beta\Lambda\eta<\mu\kappa d(\delta+\alpha_1+\mu)$;

\item Unstable, if $\beta\Lambda\eta>\mu\kappa d(\delta+\alpha_1+\mu)$.
\end{enumerate}
Moreover, if $\beta\Lambda\eta=\mu\kappa d(\delta+\alpha_1+\mu)$, 
then a critical case occurs.
\end{theorem}

\begin{proof}
The characteristic polynomial associated with the linearized 
system of model \eqref{ModeloColera} is given by
\begin{align*}
p(\chi)=\det(F_0-V_0-\chi I_5).
\end{align*} 
In order to compute the roots of polynomial $p$, we have that
\begin{equation*}
\left|\begin{matrix}
-\mu-\chi & 0 & 0 & \omega & -\displaystyle\frac{\beta \Lambda}{\mu\kappa} \\[0.25cm]
0 & -a_1-\chi & 0 & 0 & \displaystyle\frac{\beta \Lambda}{\mu\kappa} \\[0.25cm]
0 & \delta & -a_2-\chi & 0 & 0 \\[0.25cm]
0 & 0 & \varepsilon & -a_3-\chi & 0 \\[0.25cm]
0 & \eta & 0 & 0 & -d-\chi
\end{matrix}\right|=0,
\end{equation*}
that is,
$$
\chi=-\mu \vee \chi=-a_2 \vee \chi=-a_3 \vee \tilde{p}(\chi)=\chi^2+(a_1+d)\chi+a_1d-\frac{\beta\Lambda\eta}{\mu\kappa}=0.
$$
By Routh's criterion (see, e.g., p.~55--56 of \cite{Olsder}), if all coefficients 
of polynomial $\tilde{p}(\chi)$ have the same signal, then the roots 
of $\tilde{p}(\chi)$ have negative real part and, consequently, 
the DFE $E^0$ is locally asymptotic stable. The coefficients 
of $\tilde{p}(\chi)$ are $\tilde{p}_1=1>0$, $\tilde{p}_2=a_1+d>0$ 
and $\tilde{p}_3=a_1d-\frac{\beta\Lambda\eta}{\mu\kappa}$. 
Therefore, the DFE $E^0$ is
\begin{enumerate}
\item Locally asymptotic stable, if
$a_1d-\frac{\beta\Lambda\eta}{\mu\kappa}>0
\Leftrightarrow
\beta\Lambda\eta<\mu\kappa d(\delta+\alpha_1+\mu)$;
	
\item Unstable, if $a_1d-\frac{\beta\Lambda\eta}{\mu\kappa}
<0\Leftrightarrow\beta\Lambda\eta>\mu\kappa d(\delta+\alpha_1+\mu)$.
\end{enumerate}
A critical case is obtained if $a_1d=\frac{\beta\Lambda\eta}{\mu\kappa}
\Leftrightarrow\beta\Lambda\eta=\mu\kappa d(\delta+\alpha_1+\mu)$.
\end{proof}

Next we prove the existence of an endemic equilibrium 
when the basic reproduction number \eqref{R0} is greater than one.

\begin{proposition}[Endemic equilibrium]
\label{prop:EE}
If $R_0>1$, then the model \eqref{ModeloColera} 
has an endemic equilibrium given by
\begin{equation}
\label{EndemicEquilibrium}
E^*=(S^*,I^*,Q^*,R^*,B^*)
=\left(\frac{\Lambda a_1a_2a_3}{D},\frac{\Lambda a_2a_3\lambda^*}{D},
\frac{\Lambda\delta a_3\lambda^*}{D},\frac{\Lambda \delta\varepsilon\lambda^*}{D},
\frac{\Lambda\eta a_2a_3\lambda^*}{Dd}\right),
\end{equation}
where $a_1=\delta+\alpha_1+\mu$, $a_2=\varepsilon+\alpha_2+\mu$, 
$a_3=\omega+\mu$, $D=a_1a_2a_3(\lambda^*+\mu)-\delta\varepsilon\omega\lambda^*$ 
and $\lambda^*=\frac{\beta B^*}{\kappa+B^*}$.
\end{proposition}

\begin{proof}
In order to exist disease, the rate of infection must satisfy 
the inequality $\frac{\beta B(t)}{\kappa+B(t)}>0$. Considering 
that $E^*=(S^*,I^*,Q^*,R^*,B^*)$ is an endemic equilibrium 
of \eqref{ModeloColera}, let us define $\lambda^*$ to be the 
rate of infection in the presence of disease, that is,
\begin{align*}
\lambda^*=\frac{\beta B^*}{\kappa+B^*}.
\end{align*}
Using \eqref{eq:as}, considering 
$D=a_1a_2a_3(\lambda^*+\mu)-\delta\varepsilon\omega\lambda^*$
and setting the left-hand side of the equations of \eqref{ModeloColera} 
equal to zero, we obtain the endemic equilibrium \eqref{EndemicEquilibrium}.
Thus, we can compute $\lambda^*$:
\begin{equation*}
\begin{split}
\lambda^*=\frac{\beta B^*}{\kappa+B^*}
=\frac{\beta\Lambda\eta a_2a_3\lambda^*}{\kappa Dd
+\Lambda\eta a_2a_3\lambda^*}
&\Leftrightarrow
\lambda^*\left(1-\frac{\beta\Lambda\eta a_2a_3}{\kappa Dd
+\Lambda\eta a_2a_3\lambda^*}\right)=0\\
&\Leftrightarrow
\lambda^*\left(\frac{\kappa Dd+\Lambda\eta a_2a_3\lambda^*
-\beta\Lambda\eta a_2a_3}{\kappa Dd+\Lambda\eta a_2a_3\lambda^*}\right)=0.
\end{split}
\end{equation*}
The solution $\lambda^*=0$ does not make sense in this context. 
Therefore, we only consider the solution of
$\kappa Dd+\Lambda\eta a_2a_3\lambda^*-\beta\Lambda\eta a_2a_3=0$. 
We have,
\begin{equation*}
\begin{split}
\kappa &Dd+\Lambda\eta a_2a_3\lambda^*-\beta\Lambda\eta a_2a_3=0\\
&\Leftrightarrow
\kappa(a_1a_2a_3(\lambda^*+\mu)-\delta\varepsilon\omega\lambda^*)d
+\Lambda\eta a_2a_3\lambda^*-\beta\Lambda\eta a_2a_3=0\\
&\Leftrightarrow
(\kappa(a_1a_2a_3-\delta\varepsilon\omega)d+\Lambda\eta a_2a_3)\lambda^*
=-\kappa a_1a_2a_3\mu d+\beta\Lambda\eta a_2a_3\\
&\Leftrightarrow
\lambda^*=\frac{a_2a_3(\beta\Lambda\eta-\mu\kappa d a_1)}{\kappa(a_1a_2a_3
-\delta\varepsilon\omega)d+\Lambda\eta a_2a_3}
=\frac{\mu\kappa d a_1a_2a_3(R_0-1)}{\kappa(a_1a_2a_3
-\delta\varepsilon\omega)d+\Lambda\eta a_2a_3}.
\end{split}
\end{equation*}
Note that
$a_1a_2a_3-\delta\varepsilon\omega=(\delta+\alpha_1+\mu)(\varepsilon
+\alpha_2+\mu)(\omega+\mu)-\delta\varepsilon\omega>0$
because $\alpha_1, \alpha_2 \geq 0$ and $\mu>0$. Furthermore, 
since $\kappa$, $d$, $\Lambda$, $\eta>0$, we have that
$\mu\kappa d a_1a_2a_3$ and 
$\kappa(a_1a_2a_3-\delta\varepsilon\omega)d+\Lambda\eta a_2a_3$ 
are positive. Concluding, if $R_0>1$, then $\lambda^*>0$ and, consequently, 
the model \eqref{ModeloColera} has an endemic equilibrium 
given by \eqref{EndemicEquilibrium}.
\end{proof}

We end this section by proving the local stability of the endemic equilibrium $E^*$. 
Our proof is based on the center manifold theory \cite{Carr}, 
as described in Theorem~4.1 of \cite{Castillo}.

\begin{theorem}[Local asymptotic stability of the endemic equilibrium \eqref{EndemicEquilibrium}] 
The endemic equilibrium $E^*$ of model \eqref{ModeloColera}
(Proposition~\ref{prop:EE}) is locally asymptotic stable for $R_0$ 
(Proposition~\ref{prop:R0}) near $1$.
\end{theorem}

\begin{proof}
To apply the method described in Theorem~4.1 of \cite{Castillo}, 
we consider a change of variables. Let 
\begin{equation}
\label{eq:def:X}
X = (x_1, x_2, x_3, x_4, x_5) = (S, I, Q, R, B).
\end{equation} 
Consequently, we have that the total number of individuals 
is given by $N=\sum_{i=1}^{4}x_i.$ Thus, the model 
\eqref{ModeloColera} can be written as follows:
\begin{align}
\label{ModeloColeraX}
\begin{cases}
x_1'(t)=f_1=\Lambda-\displaystyle\frac{\beta x_5(t)}{\kappa+x_5(t)}x_1(t)+\omega x_4(t)-\mu x_1(t)\\
x_2'(t)=f_2=\displaystyle\frac{\beta x_5(t)}{\kappa+x_5(t)}x_1(t)-(\delta+\alpha_1+\mu)x_2(t)\\
x_3'(t)=f_3=\delta x_2(t)-(\varepsilon+\alpha_2+\mu)x_3(t)\\
x_4'(t)=f_4=\varepsilon x_3(t)-(\omega+\mu)x_4(t)\\
x_5'(t)=f_5=\eta x_2(t)-dx_5(t).
\end{cases}
\end{align}
Choosing $\beta^*$ as bifurcation parameter and solving for $\beta$ from $R_0=1$, 
we obtain that
\begin{align*}
\beta^*=\frac{\mu\kappa d(\delta+\alpha_1+\mu)}{\Lambda\eta}.
\end{align*}
Considering $\beta=\beta^*$, the Jacobian of the system \eqref{ModeloColeraX} 
evaluated at $E^0$ is given by 
\begin{align*}
J_0^*=\left[
\begin{matrix}
-\mu & 0 & 0 & \omega & -\displaystyle\frac{a_1d}{\eta} \\
0 & -a_1 & 0 & 0 & \displaystyle\frac{a_1d}{\eta} \\
0 & \delta & -a_2 & 0 & 0 \\
0 & 0 & \varepsilon & -a_3 & 0 \\
0 & \eta & 0 & 0 & -d \\
\end{matrix}
\right].
\end{align*}
The eigenvalues of $J_0^*$ are $-d-a_1$, $-a_2$, $-a_3$, $-\mu$ and $0$. 
We conclude that zero is a simple eigenvalue of $J_0^*$ and all other 
eigenvalues of $J_0^*$ have negative real parts. Therefore, the center manifold
theory \cite{Carr} can be applied to study the dynamics of \eqref{ModeloColeraX} 
near $\beta=\beta^*$. Theorem 4.1 in \cite{Castillo} is used to show 
the local asymptotic stability of the endemic equilibrium point of \eqref{ModeloColeraX}, 
for $\beta$ near $\beta^*$. The Jacobian $J_0^*$ has, respectively, a right 
eigenvector and a left eigenvector (associated with the zero eigenvalue), 
$w=\left[
\begin{matrix}
w_1 & w_2 & w_3 & w_4 & w_5
\end{matrix}
\right]^T$ 
and $v=\left[
\begin{matrix}
v_1 & v_2 & v_3 & v_4 & v_5
\end{matrix}
\right]^T$, given by
\begin{align*}
w&=\left[\begin{matrix}\left(\displaystyle\frac{\delta\varepsilon\omega}{a_2a_3}
-a_1\right)\displaystyle\frac{1}{\mu} & 1 & \displaystyle\frac{\delta}{a_2} & \displaystyle\frac{\delta\varepsilon}{a_2a_3} 
& \displaystyle\frac{\eta}{d}
\end{matrix}
\right]^Tw_2
\end{align*}
and
\begin{align*}
v=\left[\begin{matrix}0 & 1 & 0 & 0 & 
\displaystyle\frac{a_1}{\eta}\end{matrix}\right]^Tv_2.
\end{align*}
Remember that $f_l$ represents the right-hand side of the 
$l$th equation of the system \eqref{ModeloColeraX} 
and $x_l$ is the state variable whose derivative is given 
by the $l$th equation for $l=1, \ldots, 5$. 
The local stability near the bifurcation point $\beta=\beta^*$ 
is determined by the signs of two associated constants $a$ and $b$ defined by
\begin{align*}
a=\sum_{k,i,j=1}^{5}v_kw_iw_j\left[\frac{\partial^2f_k}{\partial x_i
\partial x_j}(E^0)\right]_{\beta=\beta^*}
\end{align*}
and
\begin{align*}
b=\sum_{k,i=1}^{5}v_kw_i\left[\frac{\partial^2f_k}{\partial x_i
\partial \phi}(E^0)\right]_{\beta=\beta^*}
\end{align*}
with $\phi=\beta-\beta^*$. As $v_1=v_3=v_4=0$, the nonzero partial 
derivatives at the disease free equilibrium $E^0$ are given by
\begin{align*}
\frac{\partial^2f_2}{\partial x_1\partial x_5}
=\frac{\partial^2f_2}{\partial x_5\partial x_1}
=\frac{\beta}{\kappa}
\ \text{ and }\ \frac{\partial^2f_2}{\partial x_5^2}
=-\frac{2\beta\Lambda}{\kappa^2\mu}.
\end{align*}
Therefore, the constant $a$ is
\begin{align*}
a=\frac{2\eta\beta^*}{d\kappa\mu}\left(\frac{\delta\varepsilon\omega
-a_1a_2a_3}{a_2a_3}-\frac{\Lambda\eta}{d\kappa}\right)v_2w_2^2<0.
\end{align*}
Furthermore, we have that
\begin{align*}
b=&\sum_{i=1}^{5}\left(v_2w_i\left[
\frac{\partial^2f_2}{\partial x_i\partial\beta}(E^0)\right]_{\beta=\beta^*}
+v_5w_i\left[\frac{\partial^2f_5}{\partial 
x_i\partial\beta}(E^0)\right]_{\beta=\beta^*}\right)\\
=&\sum_{i=1}^{5}v_2w_i\left[\frac{\partial}{\partial x_i}\left(\frac{x_1x_5}{\kappa+x_5}\right)(E^0)\right]_{\beta=\beta^*}\\
=&\frac{v_2w_5\Lambda}{\mu\kappa}\\
=&\frac{\Lambda\eta}{\mu\kappa d}v_2w_2>0.
\end{align*}
Thus, by Theorem~4.1 in \cite{Castillo}, we conclude 
that the endemic equilibrium $E^*$ of \eqref{ModeloColera} 
is locally asymptotic stable for a value 
of the basic reproduction number $R_0$ close to $1$.
\end{proof}

In Section~\ref{Sec:model} we propose a mathematical model,
while in Section~\ref{Sec:mod:analysis} we show that it is both mathematically 
and epidemiologically well posed for the reality under investigation.
These two sections give a model to study and understand a certain reality,
but do not allow to interfere and manipulate it. This is done in 
Section~\ref{sec:ocp}, where we introduce a control that allow us to decide how many
individuals move to quarantine. Naturally, the question is then to know how to choose
such control in an optimal way. For that, we use the theory of optimal control 
\cite{Pontryagin_et_all_1962}. After the theoretical study of the optimal control 
problem done in Section~\ref{sec:ocp}, we provide in Section~\ref{sec:num:simu} 
numerical simulations for the cholera outbreak, that occurred 
in Haiti in 2010, showing how we can manipulate and improve the reality.


\section{Optimal Control Problem}
\label{sec:ocp}

\indent 

In this section, we propose and analyze an optimal control 
problem applied to cholera dynamics described by model \eqref{ModeloColera}.
We add to model \eqref{ModeloColera} a control function $u(\cdot)$ 
that represents the fraction of infected individuals $I$ that 
are submitted to treatment in quarantine until complete recovery.
Given the meaning of the control $u$, it is natural
that the control takes values in the closed set $[0, 1]$:
$u = 0$ means ``no control measure'' and $u = 1$ means all infected people
are put under quarantine. Only values of $u$ on the interval
$[0, 1]$ make sense. The model with control is given by the following system
of nonlinear ordinary differential equations:
\begin{equation}
\label{ModeloColeraControlo}
\begin{cases}
S'(t)=\Lambda-\displaystyle\frac{\beta B(t)}{\kappa+B(t)}S(t)+\omega R(t)-\mu S(t),\\
I'(t)=\displaystyle\frac{\beta B(t)}{\kappa+B(t)}S(t)-\delta u(t) I(t)-(\alpha_1+\mu)I(t),\\
Q'(t)=\delta u(t)I(t)-(\varepsilon+\alpha_2+\mu)Q(t),\\
R'(t)=\varepsilon Q(t)-(\omega+\mu)R(t),\\
B'(t)=\eta I(t)-dB(t).
\end{cases}
\end{equation}
The set $\mathcal{X}$ of admissible trajectories is given by
\begin{equation*}
\mathcal{X} = \left\{X(\cdot) \in W^{1,1}([0,T];\R^5) \, | \,  
\eqref{eq:init:cond} \text{ and } \eqref{ModeloColeraControlo} 
\text{ are satisfied}\right\}
\end{equation*}
with $X$ defined in \eqref{eq:def:X}
and the admissible control set $\mathcal{U}$ is given by
\begin{equation*}
\mathcal{U} = \left\{ u(\cdot) \in L^{\infty}([0, T]; \mathbb{R}) \,
| \,  0 \leq u (t) \leq 1 ,  \, \forall \, t \in [0, T] \, \right\} .
\end{equation*}
We consider the objective functional  
\begin{equation}
\label{costfunction}
J(X(\cdot),u(\cdot)) = \int_0^{T} \left( I(t) + B(t)
+ \frac{W}{2}u^2(t) \right) dt \, ,
\end{equation}
where the positive constant $W$ is a measure of the cost 
of the interventions associated with the control $u$, 
that is, associated with the treatment of infected individuals 
keeping them in quarantine during all the treatment period. 
Our aim is to minimize the number of infected individuals 
and the bacterial concentration, as well as the cost of 
interventions associated with the control treatment through quarantine. 
The optimal control problem consists of determining the vector function
$X^\diamond(\cdot) = \left(S^\diamond(\cdot), 
I^\diamond(\cdot), Q^\diamond(\cdot), 
R^\diamond(\cdot), B^\diamond(\cdot)\right)
\in \mathcal{X}$ associated with an admissible control
$u^\diamond(\cdot) \in \mathcal{U}$ on the time interval $[0, T]$,
minimizing the cost functional \eqref{costfunction}, \textrm{i.e.},
\begin{equation}
\label{mincostfunct}
J(X^\diamond(\cdot),u^\diamond(\cdot)) 
= \min_{(X(\cdot),u(\cdot))
\in\mathcal{X}\times\mathcal{U}} J(X(\cdot),u(\cdot)).
\end{equation}

The existence of an optimal control $u^\diamond(\cdot)$
comes from the convexity of the cost functional \eqref{costfunction}
with respect to the controls and the regularity of the system
\eqref{ModeloColeraControlo}: see, \textrm{e.g.},
\cite{Cesari_1983,Fleming_Rishel_1975}.

\begin{remark}
In optimal control theory and in its many applications is standard 
to consider objective functionals with integrands that are convex 
with respect to the control variables \cite{MR2316829}. 
Such convexity easily ensures the existence and the regularity 
of solution to the problem \cite{MR2099056} as well as  
good performance of numerical methods \cite{MR3388961}. 
In our case, we considered a quadratic expression of the
control in order to indicate nonlinear costs potentially
arising at high treatment levels, as proposed in \cite{Neilan}.
\end{remark}

According to the Pontryagin Maximum Principle \cite{Pontryagin_et_all_1962},
if $u^\diamond(\cdot) \in \mathcal{U}$ is optimal for problem
\eqref{mincostfunct} with fixed final time $T$, then there exists
a nontrivial absolutely continuous mapping $\lambda : [0, T] \to \mathbb{R}^5$,
$\lambda(t) = \left(\lambda_1(t), \lambda_2(t), \lambda_3(t), 
\lambda_4(t), \lambda_5(t)\right)$, called the \emph{adjoint vector}, such that
\begin{enumerate}
\item the control system
\begin{equation*}
S' = \frac{\partial H}{\partial \lambda_1} \, , \quad
I' = \frac{\partial H}{\partial \lambda_2} \, , \quad
Q' = \frac{\partial H}{\partial \lambda_3} \, , \quad
R' = \frac{\partial H}{\partial \lambda_4} \, , \quad
B' = \frac{\partial H}{\partial \lambda_5};
\end{equation*}
\item the adjoint system
\begin{equation}
\label{adjsystemPMP}
\lambda_1' = -\frac{\partial H}{\partial S} \, , \quad
\lambda_2' = -\frac{\partial H}{\partial I} \, , \quad
\lambda_3' = -\frac{\partial H}{\partial Q} \, , \quad
\lambda_4' = -\frac{\partial H}{\partial R} \, , \quad
\lambda_5' = -\frac{\partial H}{\partial B};
\end{equation}

\item and the minimization condition
\begin{equation}
\label{maxcondPMP}
H(X^\diamond(t), u^\diamond(t), \lambda^\diamond(t))
= \min_{0 \leq u \leq 1}
H(X^\diamond(t), u, \lambda^\diamond(t))
\end{equation}
hold for almost all $t \in [0, T]$, where the function $H$ defined by
\begin{equation*}
\begin{split}
H= H(X, u, \lambda)
&=I + B  + \frac{W}{2}u^2
+ \lambda_1 \left( \Lambda-\displaystyle\frac{\beta B}{\kappa+B}S
+\omega R-\mu S \right)\\
&+ \lambda_2 \left( \displaystyle\frac{\beta B}{\kappa+B}S
-\delta u I-(\alpha_1+\mu)I \right)
+ \lambda_3 \left(\delta u I
-(\varepsilon+\alpha_2+\mu)Q \right)\\
&+ \lambda_4 \left( \varepsilon Q-(\omega+\mu)R \right)
+ \lambda_5 \left( \eta I-d B \right)
\end{split}
\end{equation*}
is called the \emph{Hamiltonian}. 

\item Moreover, the following transversality conditions also hold:
\begin{equation}
\label{eq:TC:PMP}
\lambda_i(T) = 0, \quad
i =1,\ldots, 5.
\end{equation}
\end{enumerate}

\begin{theorem}
The optimal control problem \eqref{mincostfunct} 
with fixed final time $T$ admits a unique optimal solution 
$\left(S^\diamond(\cdot), I^\diamond(\cdot), Q^\diamond(\cdot), 
R^\diamond(\cdot), B^\diamond(\cdot)\right)$ associated 
with an optimal control $u^\diamond(t)$ for $t \in [0, T]$.
Moreover, there exist adjoint functions 
$\lambda_i^\diamond(\cdot)$, $i = 1, \ldots, 5$, such that
\begin{equation}
\label{adjoint_function}
\begin{cases}
\lambda^{\diamond'}_1(t) 
= \lambda_1^\diamond(t)\left(\displaystyle{
\frac{\beta B^\diamond(t)}{\kappa+B^\diamond(t)}}+\mu \right) 
-{\lambda_2^\diamond(t) 
\displaystyle\frac{\beta B^\diamond(t)}{\kappa+B^\diamond(t)}} \, ,\\[0.1 cm]
\lambda^{\diamond'}_2(t) = -1 +{\lambda_2^\diamond(t)}\left(u^\diamond(t)\delta 
+ \alpha_1 + \mu \right) -\lambda_3^\diamond(t) u^\diamond(t)\delta 
-\lambda_5^\diamond(t)\eta \, , \\[0.1 cm]
\lambda^{\diamond'}_3(t) 
= \lambda_3^\diamond(t)\left( \varepsilon + \alpha_2 + \mu \right) 
-\lambda_4^\diamond(t)\varepsilon \, , \\[0.1 cm]
\lambda^{\diamond'}_4(t) = - \lambda_1^\diamond(t)\omega 
+ \lambda_4^\diamond(t)\left(\omega +\mu \right) \, , \\[0.1 cm]
\lambda^{\diamond'}_5(t) =  -1 + \lambda_1^\diamond(t)
\displaystyle{\frac {\beta\kappa S^\diamond(t)}{(\kappa+B^\diamond(t))^2}} 
- \lambda_2^\diamond(t)\displaystyle{\frac{\beta\kappa 
S^\diamond(t)}{(\kappa+B^\diamond(t))^2}}
+ \lambda_5^\diamond(t)d   \, ,
\end{cases}
\end{equation}
with transversality conditions
\begin{equation}
\label{eq:TC}
\lambda^\diamond_i(T) = 0,
\quad i=1, \ldots, 5 \, .
\end{equation}
Furthermore,
\begin{equation}
\label{optcontrols}
u^\diamond(t) = \min \left\{ \max \left\{0, 
\frac{\delta I^\diamond(t)\, \left( \lambda_2^\diamond(t) 
- \lambda_3^\diamond(t) \right)}{W} \right\}, 1 \right\} \, .
\end{equation}
\end{theorem}

\begin{proof}
Existence of an optimal solution $\left(S^\diamond, I^\diamond, 
Q^\diamond, R^\diamond, B^\diamond\right)$ associated with an 
optimal control $u^\diamond$ comes from the convexity of the 
integrand of the cost function $J$ with respect to the control 
$u$ and the Lipschitz property of the state system 
with respect to the state variables $\left(S, I, Q, R, B\right)$ 
(see, \textrm{e.g.}, \cite{Cesari_1983,Fleming_Rishel_1975}).
System \eqref{adjoint_function} is derived from the
adjoint system \eqref{adjsystemPMP},
conditions \eqref{eq:TC} from the 
transversality conditions \eqref{eq:TC:PMP},
while the optimal control \eqref{optcontrols} comes from 
the minimization condition \eqref{maxcondPMP}
of the Pontryagin Maximum Principle \cite{Pontryagin_et_all_1962}. 
For small final time $T$, the optimal control pair 
given by \eqref{optcontrols} is unique due to the boundedness 
of the state and adjoint functions and the Lipschitz property 
of systems \eqref{ModeloColeraControlo} and \eqref{adjoint_function}.
Uniqueness extends to any $T$ due to the fact that our problem
is autonomous (see \cite{SilvaTorresTBMBS} and references cited therein).
\end{proof}

In Section~\ref{sec:num:simu} we solve numerically 
the optimal control problem \eqref{mincostfunct}.


\section{Numerical Simulations}
\label{sec:num:simu}

\indent 

We start by considering in Section~\ref{sub:sec:SIB} the cholera 
outbreak that occurred in the Department of Artibonite, Haiti, 
from $1$st November 2010 to $1$st May 2011 \cite{WHO}.
Then, in Section~\ref{sub:sec:SIQRB}, we illustrate the local stability 
of the endemic equilibrium for the complete model \eqref{ModeloColera}.
Finally, in Section~\ref{sub:sec:OC}, we solve numerically the optimal 
control problem proposed and studied in Section~\ref{sec:ocp}.
Note that in all our numerical simulations the conditions 
of Lemma~\ref{lema:Inv} are satisfied.


\subsection{SIB sub-model}
\label{sub:sec:SIB}

\indent 

To approximate the real data we choose 
$\omega=\delta=\varepsilon=\alpha_2= Q(0)=R(0)=0$, 
obtaining the sub-model of \eqref{ModeloColera} given by
\begin{align}
\label{SubModeloColera}
\begin{cases}
S'(t)=\Lambda-\displaystyle\frac{\beta B(t)}{\kappa+B(t)}S(t)-\mu S(t)\\
I'(t)=\displaystyle\frac{\beta B(t)}{\kappa+B(t)}S(t)-(\alpha_1+\mu)I(t)\\
B'(t)=\eta I(t)-dB(t).
\end{cases}
\end{align}
Note that the existing data of the cholera outbreak \cite{WHO}
does not include quarantine and, consequently,
recovered individuals.
By considering the other parameter values from Table~\ref{Tab_Parameter}, 
the sub-model \eqref{SubModeloColera} approximates well the cholera outbreak 
in the Department of Artibonite, Haiti: see Figure~\ref{fig:01}. 
In this situation, the basic reproduction number \eqref{R0} is 
$R_0=35.7306$ and the endemic equilibrium \eqref{EndemicEquilibrium} is 
$$
E^*=(S^*,I^*,B^*)=(620.2829,32.2234,976.4658).
$$
\begin{figure}[ht!]
\centering
\subfloat[Infectious with symptoms $I^*$]{%
\label{FigHaitiSemQ}\includegraphics[scale=0.45]{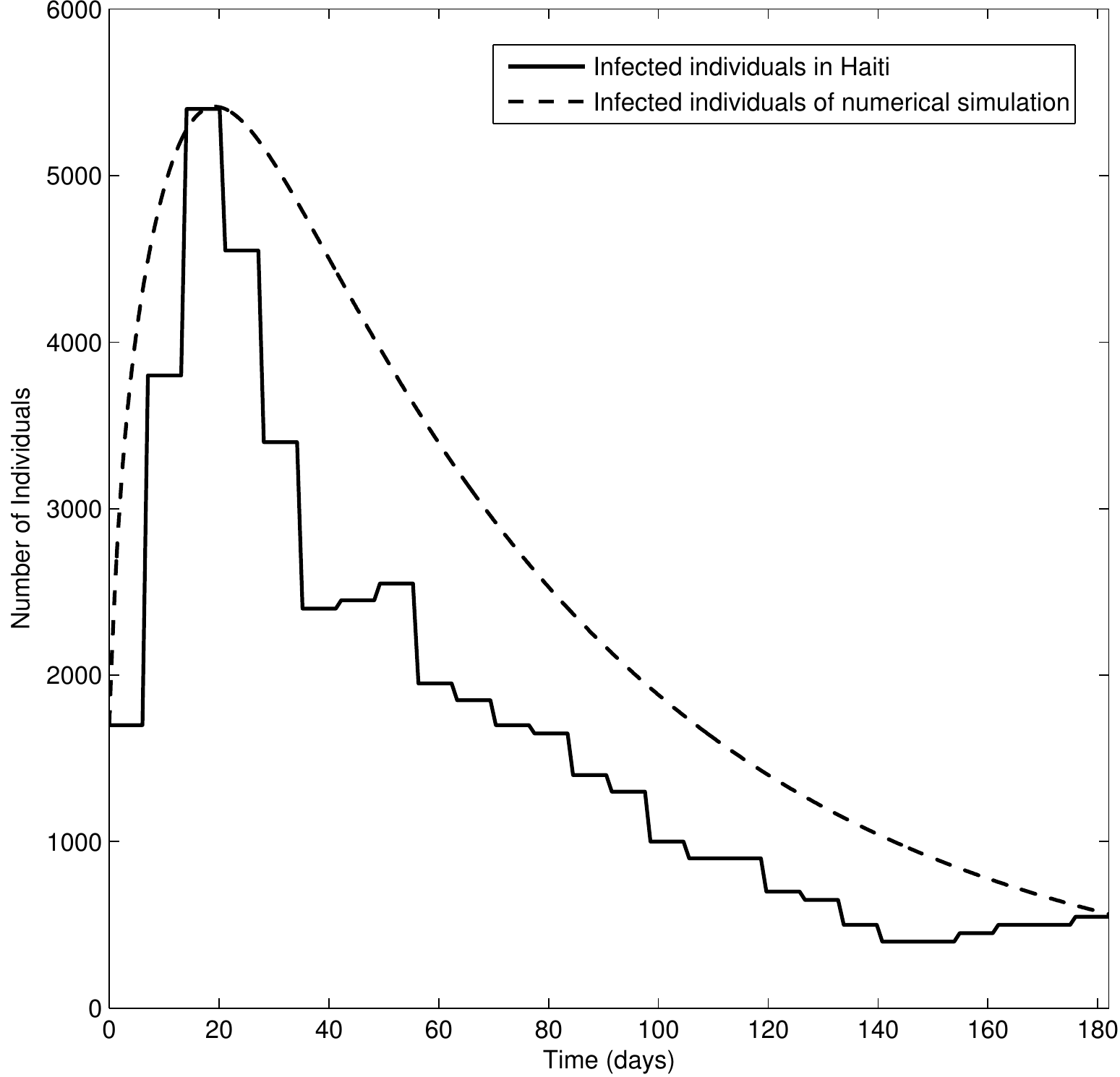}}
\qquad
\subfloat[Bacterial concentration $B^*$]{%
\label{FigHaitiSemQ:b}\includegraphics[scale=0.45]{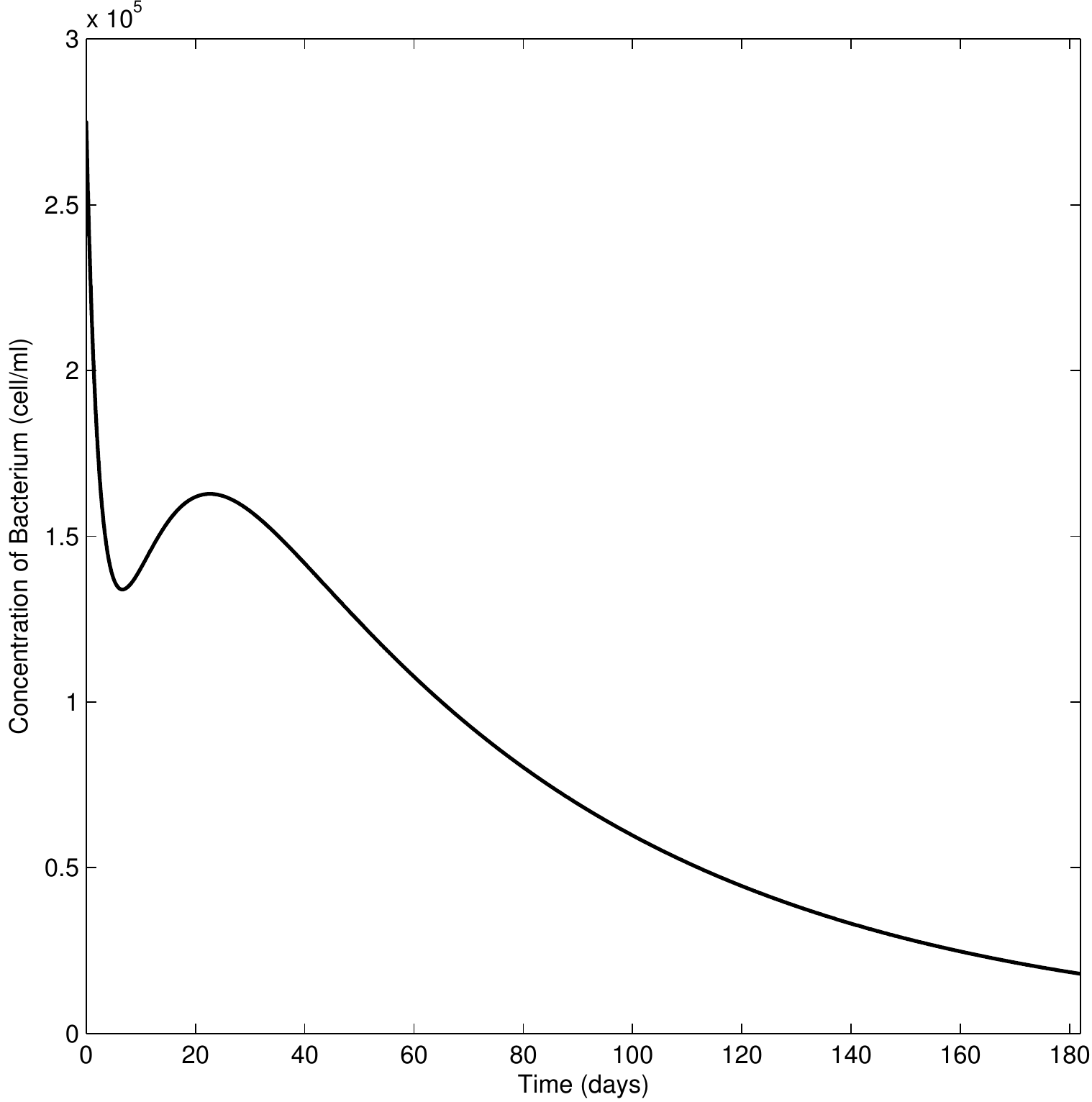}}
\caption{Solutions $I^*$ and $B^*$ predicted by our model \eqref{SubModeloColera} 
with $\omega=\delta=\varepsilon=\alpha_2=0$
and the other parameter values as given in Table~\ref{Tab_Parameter}
versus real data from the cholera outbreak 
in the Department of Artibonite, Haiti, from $1$st 
November 2010 to $1$st May 2011 (solid line in (a)).}
\label{fig:01}
\end{figure}


\subsection{Local stability of the endemic equilibrium of the SIQRB model}
\label{sub:sec:SIQRB}

\indent 

For the parameter values in Table~\ref{Tab_Parameter}, 
we have that the basic reproduction number \eqref{R0} is
$$
R_0=8.2550
$$ 
and the endemic equilibrium \eqref{EndemicEquilibrium} is 
$$
E^*=(S^*,I^*,Q^*,R^*,B^*)=(2684.3930,27.2540,6.8093,1217.7101,825.8793).
$$ 
In Figure~\ref{fig:02} we can observe agreement
between the numerical simulation of the model \eqref{ModeloColera}
and the analysis of the local stability of the 
endemic equilibrium $E^*$ done in Section~\ref{sec:3.2}.
\begin{figure}[ht!]
\centering
\subfloat[Susceptible $S^*(t)$]{%
\label{FigSuscQ}\includegraphics[scale=0.44]{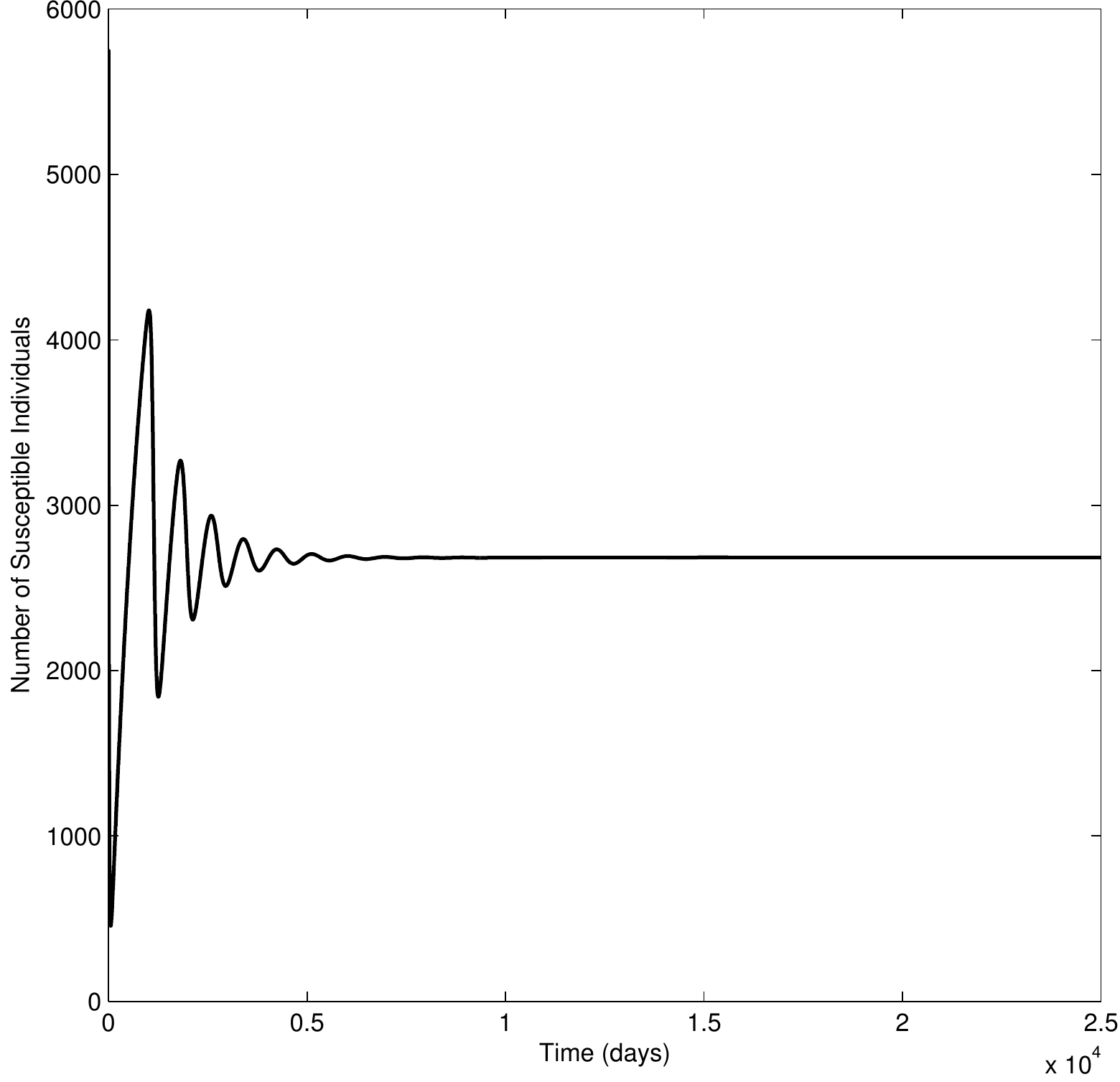}}
\qquad
\subfloat[Infectious with symptoms $I^*(t)$]{%
\label{FigInfQ}\includegraphics[scale=0.44]{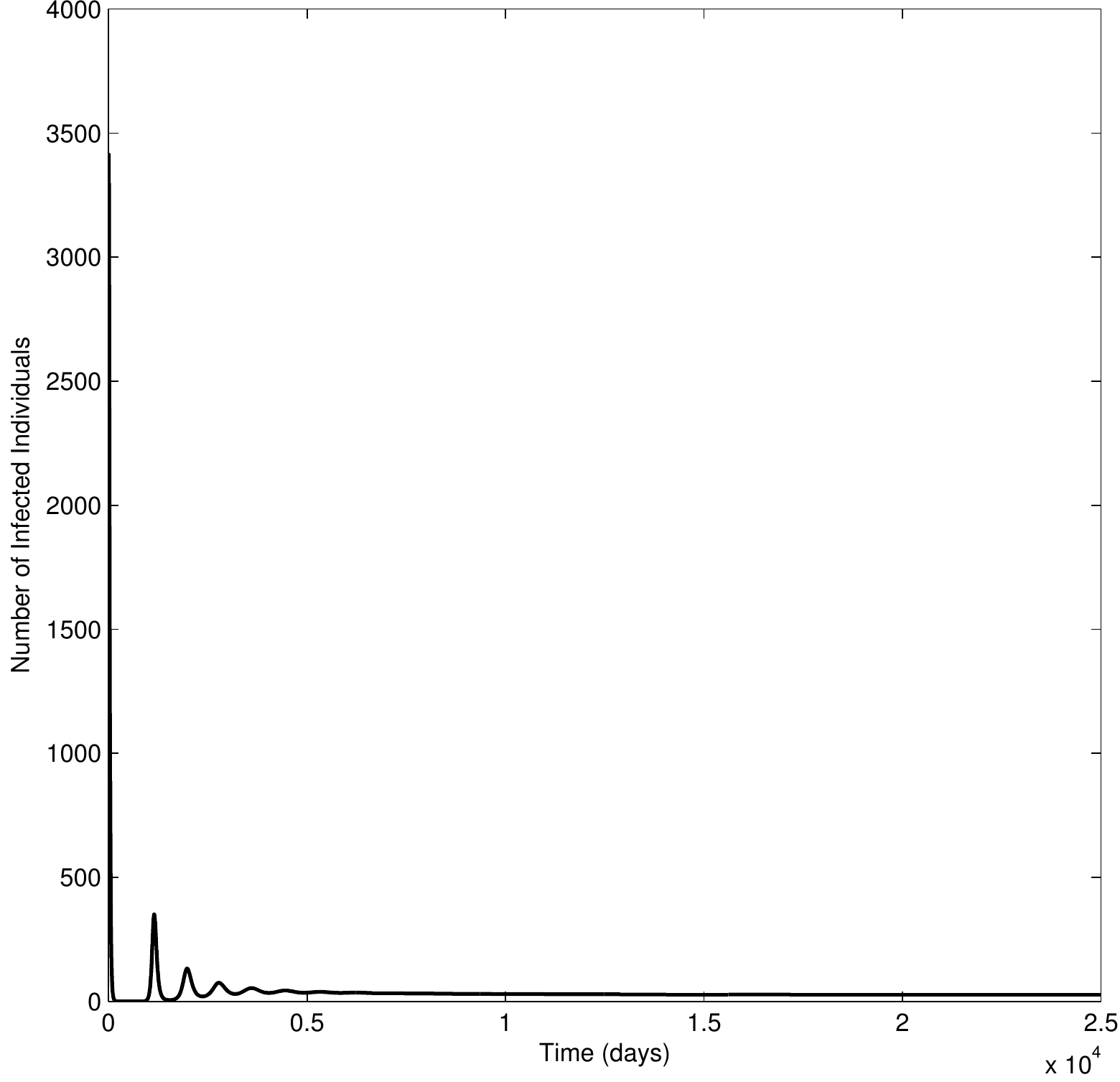}}\\
\subfloat[In treatment through quarantine $Q^*(t)$]{%
\label{FigQQ}\includegraphics[scale=0.44]{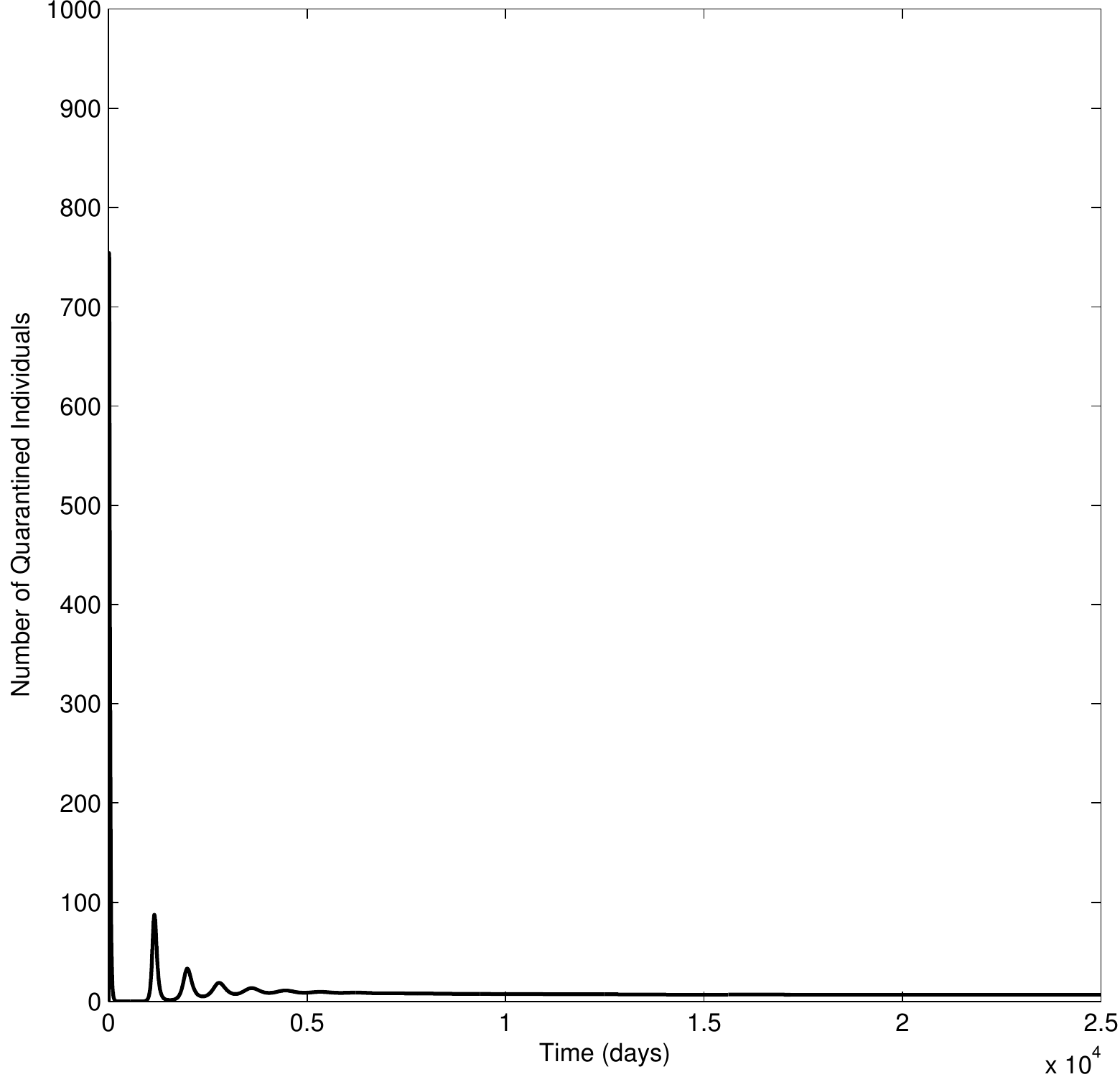}}
\qquad
\subfloat[Recovered $R^*(t)$]{%
\label{FigRQ}\includegraphics[scale=0.44]{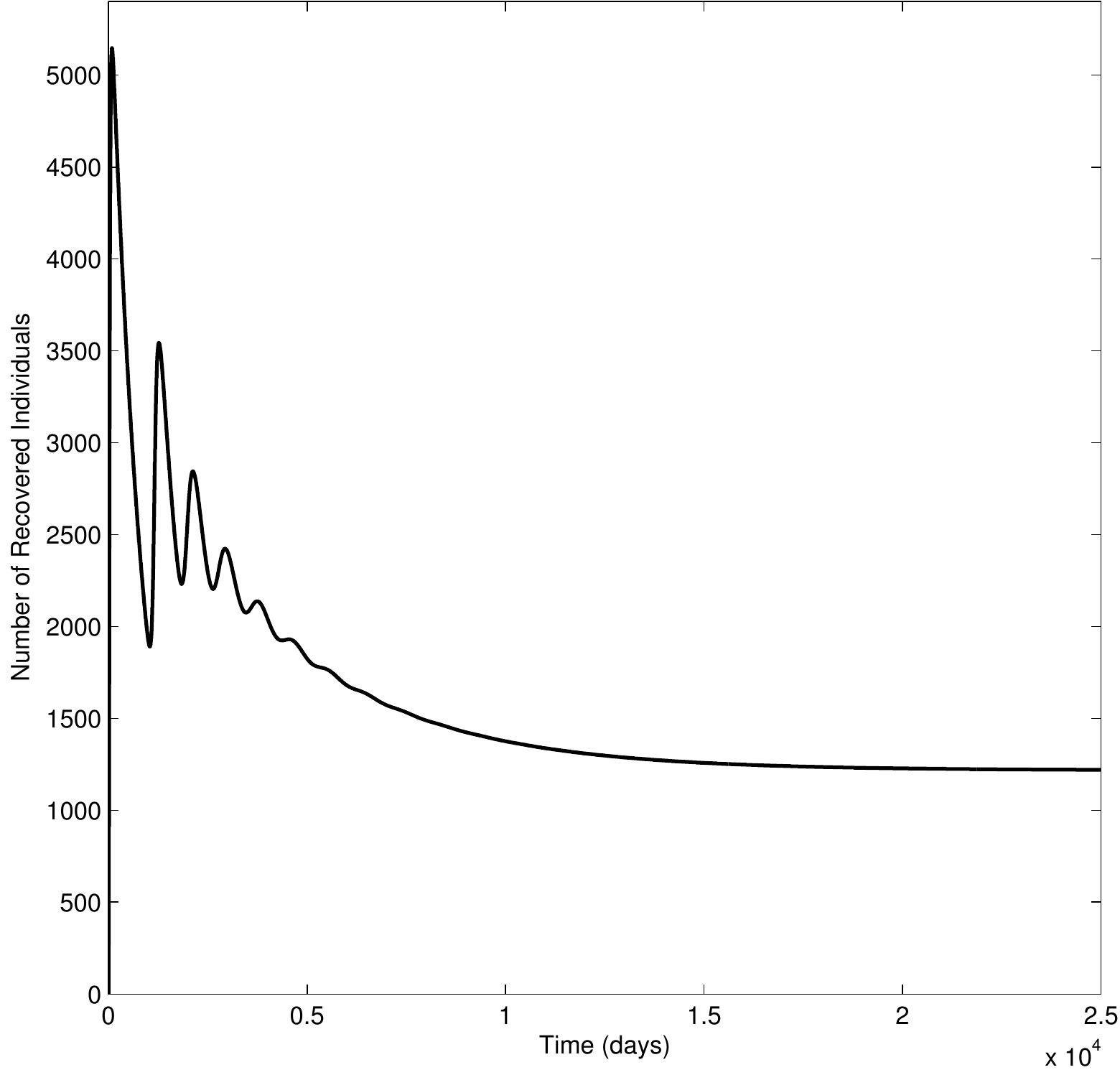}}\\
\subfloat[Bacterial concentration $B^*(t)$]{%
\label{FigBQ}\includegraphics[scale=0.44]{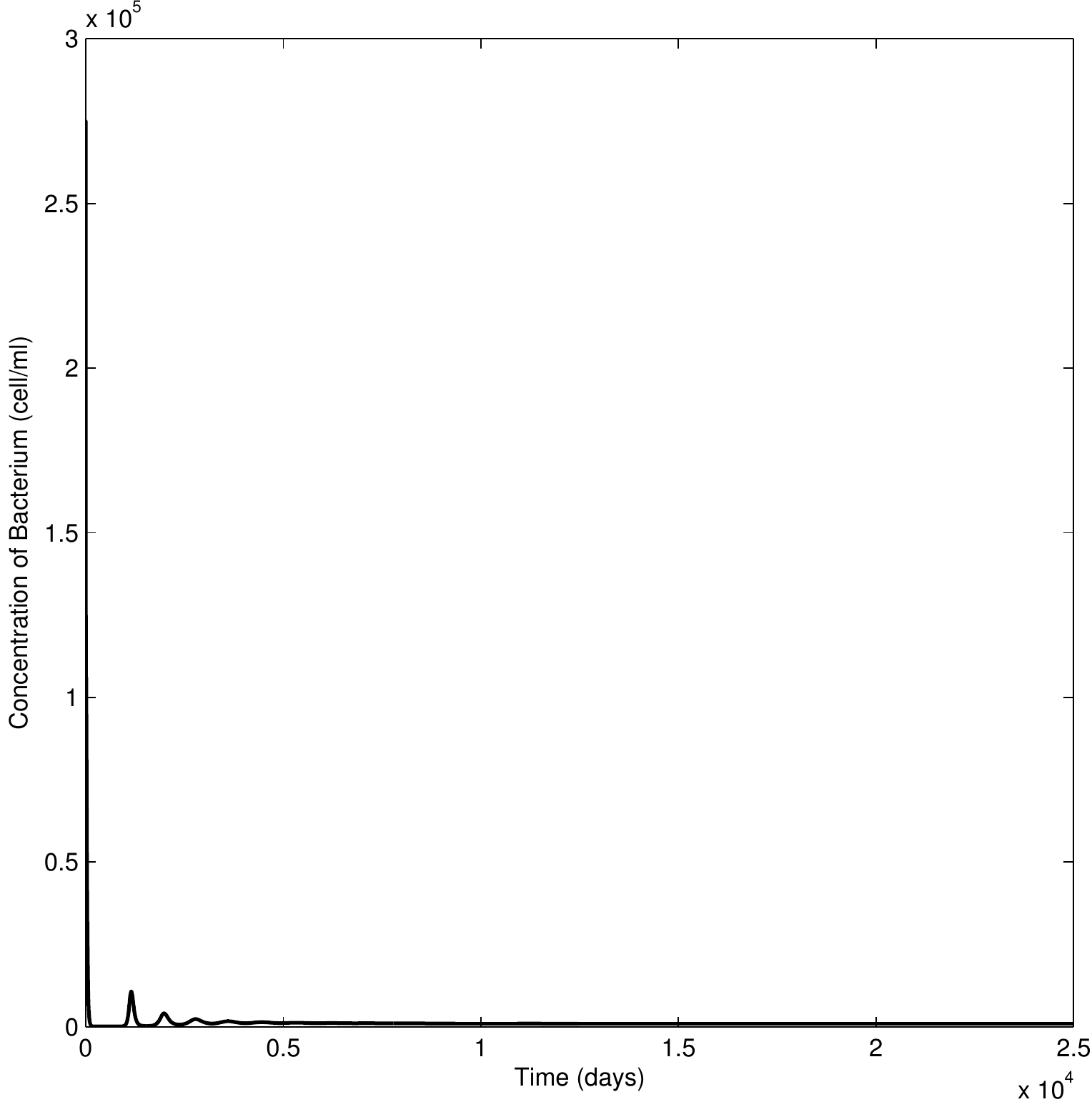}}
\caption{State trajectories of model \eqref{ModeloColera} 
for the parameter values and initial conditions of Table~\ref{Tab_Parameter}.}
\label{fig:02}
\end{figure}


\subsection{Optimal control solution}
\label{sub:sec:OC}

\indent 

We now solve the optimal control problem proposed 
in Section~\ref{sec:ocp} for $W = 2000$ \cite{fica:caro} 
and the parameter values and initial conditions 
in Table~\ref{Tab_Parameter}. 
\begin{table}[!htb]
\centering
\begin{tabular}[center]{|l|l|l|l|} \hline
\textbf{Parameter} & \textbf{Description} & \textbf{Value} & \textbf{Reference}\\ \hline \hline
\small $\Lambda$ & \small Recruitment rate & \small 24.4$N(0)$/365000 (day$^{-1}$) & \small \cite{BirthRate}\\
\small $\mu$ & \small Natural death rate & \small 2.2493$\times10^{-5}$ (day$^{-1}$) & \small \cite{DeathRate}\\
\small $\beta$ & \small Ingestion rate & \small 0.8 (day$^{-1}$)  & \small \cite{Capone}\\
\small $\kappa$ & \small Half saturation constant & \small$10^6$ (cell/ml) & \small \cite{Sanches}\\
\small $\omega$ & \small Immunity waning rate & \small 0.4/365 (day$^{-1}$) & \small \cite{Neilan}\\
\small $\delta$ & \small Quarantine rate & \small 0.05 (day$^{-1}$) & \small Assumed\\
\small $\varepsilon$ & \small Recovery rate & \small 0.2 (day$^{-1}$) & \small \cite{Mwasa}\\
\small $\alpha_1$ & \small Death rate (infected) & \small 0.015 (day$^{-1}$) & \small \cite{Mwasa}\\
\small $\alpha_2$ & \small Death rate (quarantined)& \small 0.0001 (day$^{-1}$) & \small \cite{Mwasa}\\
\small $\eta$ & \small Shedding rate (infected) & \small 10 (cell/ml day$^{-1}$ person$^{-1}$) 
& \small \cite{Capone}\\
\small $d$ & \small Bacteria death rate & \small0.33 (day$^{-1}$) & \small \cite{Capone}\\
\small $S(0)$ & \small Susceptible individuals at $t=0$ & \small 5750 (person) & \small Assumed \\
\small $I(0)$ & \small Infected individuals at $t=0$ & \small 1700 (person) & \small \cite{WHO}\\
\small $Q(0)$ & \small Quarantined individuals at $t=0$ & \small 0 (person) & \small Assumed \\
\small $R(0)$ & \small Recovered individuals at $t=0$ & \small 0 (person) & \small Assumed \\
\small $B(0)$ & \small Bacterial concentration at $t=0$ & \small$275\times10^3$ (cell/ml) 
& \small Assumed \\ \hline
\end{tabular}
\caption[]{Parameter values and initial conditions 
for the SIQRB model \eqref{ModeloColera}.}\label{Tab_Parameter}
\end{table}
The optimal control takes the maximum value for $t \in [0, 87.36]$ days. 
For $t \in\ ]87.36, 182]$, the optimal control is a decreasing function 
and at the final time we have $u^\diamond(182) \approx 0.00159$ 
(see Figure~\ref{FigControlo}). 
\begin{figure}
\centering
\includegraphics[scale=0.45]{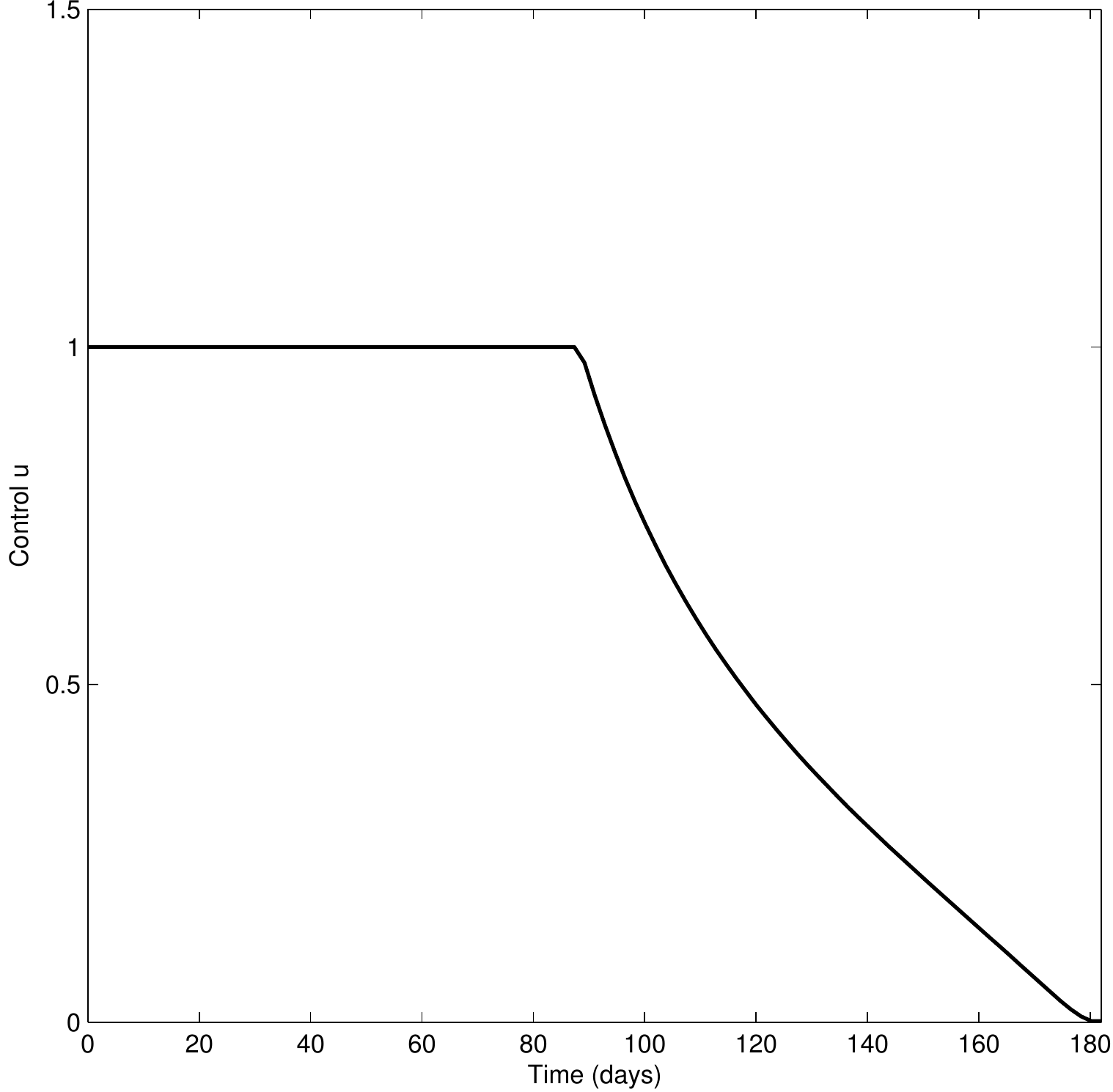}
\caption{The optimal control \eqref{optcontrols} 
for the parameter values and initial conditions 
of Table~\ref{Tab_Parameter}.}
\label{FigControlo}
\end{figure}
At the end of approximately 88 days, the number of infectious individuals 
associated with the optimal control strategy decreases from 1700 
to approximately 86 individuals and at the final time of $T = 182$ days, 
the number of infectious individuals associated with 
the optimal control is, approximately, $23$ (see Figure~\ref{FigHaitiQ:a}). 
We observe that the strategy associated with the control $u^\diamond$ 
allows an important decrease on the number of infectious individuals 
as well as on the concentration of bacteria. The maximum value 
of the number of infectious individuals also decreases significantly 
when the control strategy is applied. The optimal control implies 
a significant transfer of individuals to the recovered class. 
\begin{figure}
\centering
\subfloat[]{%
\label{FigHaitiQ:a}\includegraphics[scale=0.45]{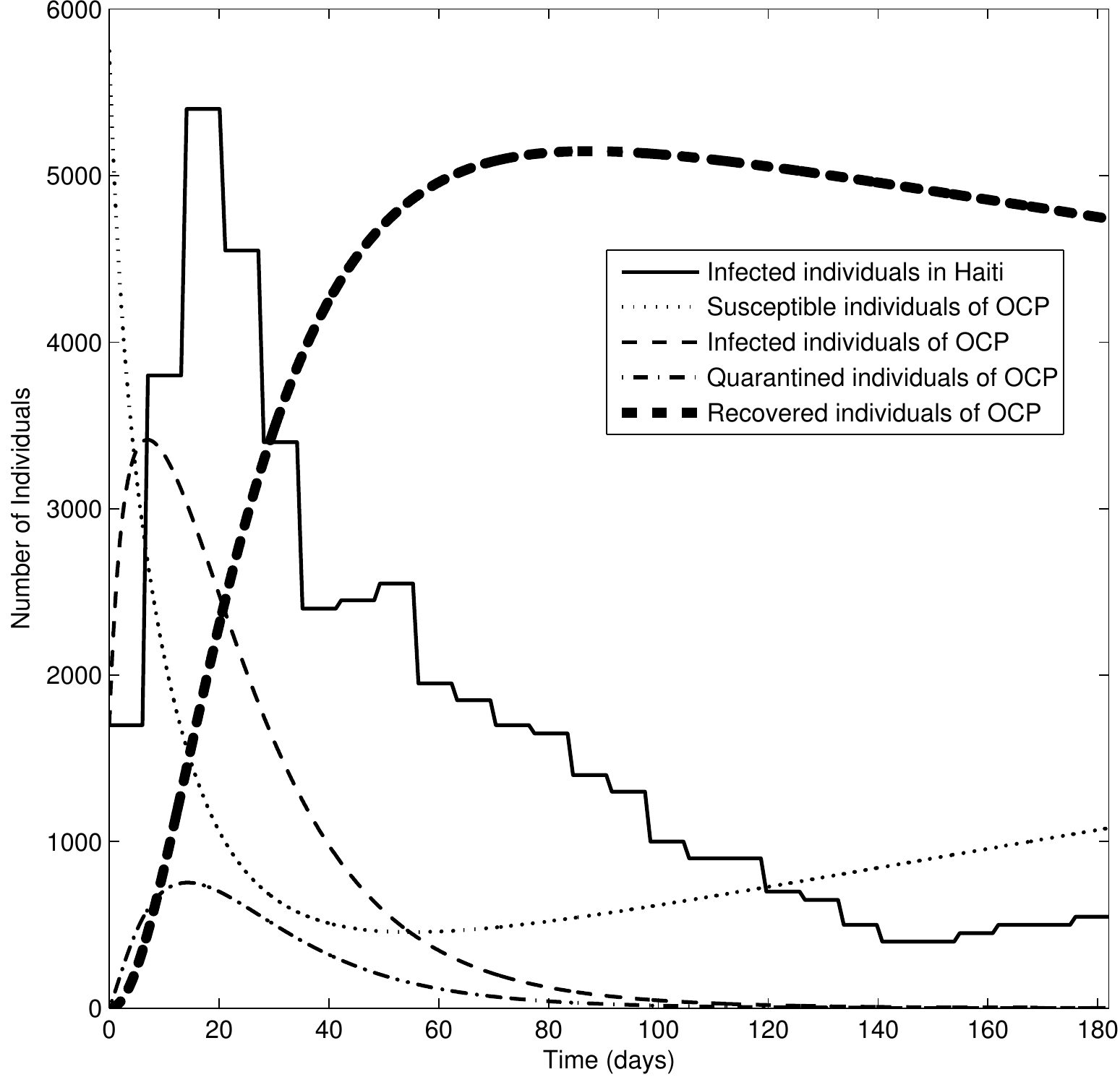}}
\qquad
\subfloat[]{%
\label{FigHaitiQ:b}\includegraphics[scale=0.45]{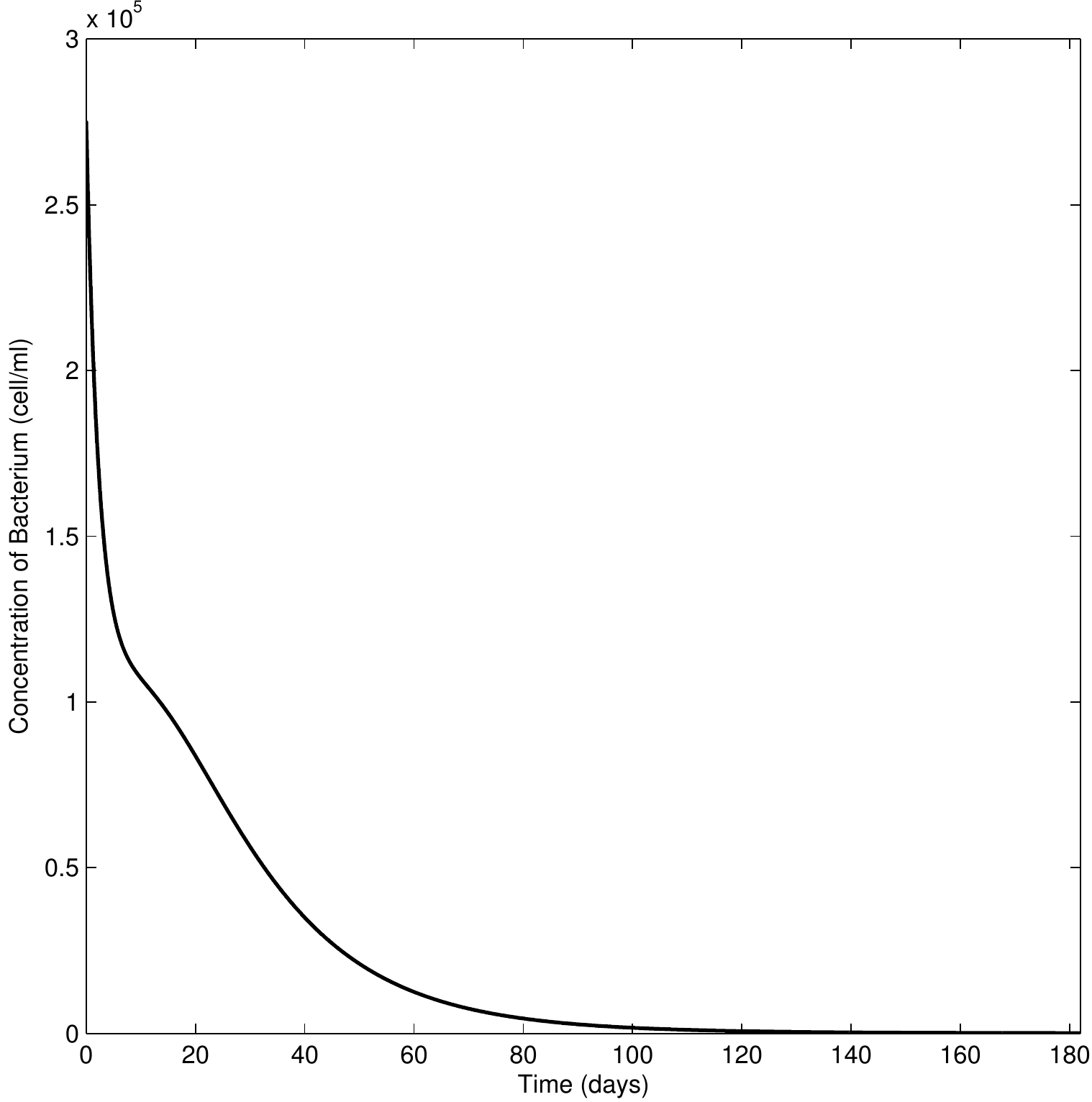}}
\caption{(a) In solid line: real data from the cholera outbreak 
in the Department of Artibonite, Haiti, from $1$st 
November 2010 to $1$st May 2011; in dashed lines 
the optimal solutions $S^\diamond$, $I^\diamond$, $Q^\diamond$ 
and $R^\diamond$ of the OCP (Optimal Control Problem) \eqref{mincostfunct}
with parameter values and initial conditions of Table~\ref{Tab_Parameter}.
(b) Optimal bacterial concentration for the optimal control problem \eqref{mincostfunct}
with parameter values and initial conditions of Table~\ref{Tab_Parameter}.}
\label{FigHaitiQ}
\end{figure}


\section{Conclusion}
\label{sec:conc}

\indent 

SIR (Susceptible--Infectious--Recovered)
type models and optimal control theory
provide powerful tools to describe and control 
infection disease dynamics \cite{MR3054565,MR3508846,Silva}.
In this paper we propose a SIQRB model for cholera transmission dynamics. 
Our model differs from the other mathematical models for cholera dynamics 
transmission in the literature, because it assumes that infectious individuals 
subject to treatment stay in quarantine during that period. Our goal is to find 
the optimal way of using quarantine with the less possible cost and, simultaneously, 
to minimize the number of infectious individuals and the bacteria concentration. 
For that we propose an optimal control problem, which is analyzed both analytically 
and numerically. The numerical simulations show that after approximately three months
($87.36$ days) the optimal strategy implies a gradual reduction of the fraction of infectious 
individuals that stay in quarantine. To be precise, by introducing the optimal 
strategy through quarantine, as a way of systematizing treatment,
one reduces the 2247 infected individuals reported by WHO in \cite{WHO}
(see Figure~\ref{FigHaitiSemQ}) to just 86 infected individuals 
(see Figure~\ref{FigHaitiQ:a}). Since quarantine implies a big economic, 
social and individual effort, it is important to know the instant of time 
from which the infectious individuals may leave quarantine 
without compromising the minimization of the number of infectious individuals 
and the bacterial concentration.

 
\section*{Acknowledgements}

\indent 

This research was supported by the
Portuguese Foundation for Science and Technology (FCT)
within projects UID/MAT/04106/2013 (CIDMA) 
and PTDC/EEI-AUT/2933/2014 (TOCCATA), funded by Project 
3599 - Promover a Produ\c{c}\~ao Cient\'{\i}fica e Desenvolvimento
Tecnol\'ogico e a Constitui\c{c}\~ao de Redes Tem\'aticas 
and FEDER funds through COMPETE 2020, Programa Operacional
Competitividade e Internacionaliza\c{c}\~ao (POCI).
Lemos-Pai\~{a}o is also supported by the Ph.D.
fellowship PD/BD/114184/2016; Silva by the post-doc 
fellowship SFRH/BPD/72061/2010. The authors are very grateful
to an anonymous Reviewer for several important
comments and questions that improved the manuscript.



\end{document}